\documentclass{svmult}

\usepackage{amssymb, latexsym, makeidx, graphicx, multicol}

\def\NN{\ensuremath{\mathbb{N}}}

\def\RR{\ensuremath{\mathbb{R}}}
\newcommand{\CC}{{\mathbb C}}

\newcommand{\V}{{\mathcal V}}
\newcommand{\B}{{\mathcal B}}

\def\I{\ensuremath{\mathcal{I}}}

\def\T{\ensuremath{\mathcal{T}}}

\def\E{\ensuremath{\mathcal{E}}}

\def\cc{\ensuremath{{\bf{c}}}}

\def\ff{\ensuremath{{\bf{f}}}}

\def\pp{\ensuremath{{\bf{p}}}}
\def\qq{\ensuremath{{\bf{q}}}}

\def\ss{\ensuremath{{\bf{s}}}}

\def\xx{\ensuremath{{\bf{x}}}}
\def\yy{\ensuremath{{\bf{y}}}}

\def\TH{\ensuremath{\textup{TH}}}

\def\Aut{\ensuremath{\textup{Aut}}}

\def\conv{\ensuremath{\textup{conv}}}
\def\cl{\ensuremath{\textup{cl}}}
\def\deg{\ensuremath{\textup{deg}}}

\def\mod{\ensuremath{\textup{mod}}}

\begin{document}

\title*{Convex Hulls of Algebraic Sets} 
\author{Jo{\~a}o Gouveia\inst{1}\and Rekha Thomas\inst{2}}
\institute{Department of Mathematics, University of Washington, Box
 354350, Seattle, WA 98195, USA, and CMUC, Department of Mathematics,
 University of Coimbra, 3001-454 Coimbra, Portugal
   \texttt{jgouveia@math.washington.edu}\and
Department of Mathematics, University of Washington, Box  354350, Seattle, WA 98195, USA \texttt{thomas@math.washington.edu}}
 
%\date{\today}

\maketitle

\begin{abstract}
 This article describes a method to compute successive convex approximations of the convex hull of a set of points in $\RR^n$ that are the solutions to a system of polynomial equations over the reals. The method relies on sums of squares of polynomials and the dual theory of moment matrices. The main feature of the technique is that all computations are done modulo the ideal generated by the polynomials defining the set to the convexified. This work was motivated by questions raised by Lov{\'a}sz concerning extensions of the theta body of a graph to arbitrary real algebraic varieties, and hence the relaxations described here are called theta bodies. The convexification process can be seen as an incarnation of Lasserre's hierarchy of convex relaxations of a semialgebraic set in $\RR^n$. 
When the defining ideal is real radical the results become especially nice. We provide several examples of the method and discuss convergence issues. Finite convergence, especially after the first step of the method, can be described explicitly for finite point sets. 
\end{abstract}

%%%%%%%%%%%%%%%%%%%%%%%%%%%%%%%%%%%%%%%%%%%%%%%%%%%%%%%%%%%%%%%%%%%%%
\section{Introduction} 

An important concern in optimization is the complete or partial knowledge of the convex hull of the set of feasible solutions to an optimization problem.  Computing convex hulls is in general a difficult task, and a classical example is the construction of the integer hull of a polyhedron which drives many algorithms in integer programming. In this article we describe a method to convexify (at least approximately), an algebraic set using semidefinite programming. 

By an algebraic set we mean a subset $S \subseteq \RR^n$ described by a finite list of polynomial equations of the form $f_1(\xx) = f_2(\xx) = \ldots = f_t(\xx) = 0$ where $f_i(\xx)$ is an element of $\RR[\xx] := \RR[x_1, \ldots, x_n]$, the polynomial ring in $n$ variables over the reals. The input to our algorithm is the ideal generated by $f_1, \ldots, f_t$, denoted as $I = \langle f_1, \ldots, f_t \rangle$, which is the set $\{ \sum_{i=1}^{t} g_i f_i \,:\, g_i \in \RR[\xx] \}$. An ideal $I \subseteq \RR[\xx]$ is a group under addition and is closed under multiplication by elements of $\RR[\xx]$. Given an ideal $I \subseteq \RR[\xx]$, its \emph{real variety}, $\V_\RR(I) = \{ \xx \in \RR^n \,:\, f(\xx) = 0 \,\,\forall \,\,f \in I \}$ is an example of an algebraic set. Given $I$, we describe a method to produce a nested sequence of convex relaxations of the closure of $\conv(\V_\RR(I))$, the convex hull of $\V_\RR(I)$, called the \emph{theta bodies} of $I$. The $k$-th theta body $\TH_k(I)$ is obtained as the projection of a \emph{spectrahedron} 
(the feasible region of a semidefinite program), and
$$\TH_1(I) \supseteq \TH_2(I) \supseteq \cdots \supseteq \TH_k(I) \supseteq \TH_{k+1}(I) \supseteq \cdots \supseteq \textup{cl}(\conv(\V_\RR(I))).$$
Of special interest to us are \emph{real radical ideals}. We define the real radical of an ideal $I$, denoted as $\sqrt[\RR]{I}$, to be the set of all polynomials $f \in \RR[\xx]$ such that $f^{2m} + \sum g_i^2 \in I$ for some $g_i \in \RR[\xx]$ and $m \in \NN$. We say that $I$ is a real radical ideal if $I = \sqrt[\RR]{I}$. Given a set $S \subseteq \RR^n$, the \emph{vanishing ideal} of $S$, denoted as $\I(S)$, is the set of all polynomials in $\RR[\xx]$ that vanish on $S$. The \emph{Real Nullstellensatz} (see Theorem~\ref{thm:real_nullstellensatz}) says that for any ideal $I$, $\I(\V_\RR(I)) = \sqrt[\RR]{I}$.

The construction of theta bodies for arbitrary ideals was motivated by a problem posed by Lov{\'a}sz. In \cite{ShannonCapacity} Lov{\'a}sz constructed the theta body of a graph, a convex relaxation of the stable set polytope of a graph which was shown later to have a description in terms of semidefinite programming. An important result in this context is that the theta body of a graph coincides with the stable set polytope of the graph if and only if the graph is perfect. Lov{\'a}sz observed that the theta body of a graph could be described in terms of sums of squares of real polynomials modulo the ideal of polynomials that vanish on the incidence vectors of stable sets. This observation naturally suggests the definition of a theta body for any ideal in $\RR[\xx]$. In fact, an easy extension of his observation leads to a hierarchy of theta bodies for all ideals as above. In \cite[Problem 8.3]{Lovasz}, Lov{\'a}sz asked to characterize all ideals that have the property that their first theta body coincides with $\textup{cl}(\conv(\V_\RR(I)))$, which was the starting point of our work. For defining ideals of finite point sets we answer this question in Section~4.

This article is organized as follows. In Section 2 we define theta bodies of an ideal in $\RR[\xx]$ in terms of sums of squares polynomials. For a general ideal $I$, we get that $\TH_k(I)$ contains the closure of the projection of a spectrahedron which is described via combinatorial moment matrices from $I$. When the ideal $I$ is real radical, we show that $\TH_k(I)$ coincides with the closure of the projected spectrahedron, and when $I$ is the defining ideal of a set of points in $\{0,1\}^n$, the closure is not needed. We establish a general relationship between the theta body sequence of an ideal $I$ and that of its real radical ideal $\sqrt[\RR]{I}$.

Section 3 gives two examples of the construction described in Section 2. As our first example, we look at the stable sets in a graph and describe the hierarchy of theta bodies that result. The first member of the hierarchy is Lov{\'a}sz's theta body of a graph. This hierarchy converges to the stable set polytope in finitely many steps as is always the case when we start with a finite set of real points. The second example is a cardiod in the plane in which case the algebraic set that is being convexified is infinite. 

In Section 4 we discuss convergence issues for the theta body sequence. When $\V_\RR(I)$ is compact, the theta body sequence is guaranteed to converge to the closure of $\conv(\V_\RR(I))$ asymptotically. We prove that when $\V_\RR(I)$ is finite, $\TH_k(I) = \textup{cl}(\conv(\V_\RR(I)))$ for some finite $k$. In the case of finite convergence, it is useful to know the specific value of $k$ for which $\TH_k(I) = \textup{cl}(\conv(\V_\RR(I)))$. This notion is called exactness and we characterize exactness for the first theta body when the set to be convexified is finite.  There are examples in which the theta body sequence does not converge to  $\textup{cl}(\conv(\V_\RR(I)))$. While a full understanding of when convergence occurs is still elusive, we describe one obstruction to finite convergence in terms of certain types of singularities of $\V_\RR(I)$. 

The last section gives more examples of theta bodies and their exactness. In particular we consider cuts in a graph and polytopes coming from the graph isomorphism question.

The core of this paper is based on results from \cite{GPT} and \cite{GLPT} which are presented here with a greater emphasis 
on geometry, avoiding some of the algebraic language in the original results. Theorems~\ref{thm:theta_inclusion}, \ref{thm:real_radical_exactness} and their corollaries are new while Theorem~\ref{thm:convex_singularity} is from \cite{GouNet}. The application of theta bodies to polytopes that arise in the graph isomorphism question is taken from \cite{DHMO}.

{\bf Acknowledgments}. Both authors were partially supported by the NSF Focused 
 Research Group grant DMS-0757371. J. Gouveia was also partially
 supported by Funda{\c c}{\~ a}o para a Ci{\^ e}ncia e Tecnologia and 
 R.R. Thomas by the Robert R. and Elaine K. Phelps Endowed Professorship.

\section{Theta bodies of polynomial ideals}

To describe the convex hull of an algebraic set, we start with a simple observation about any convex hull.
Given a set $S \subseteq \RR^n$, $\cl(\conv(S))$, the closure of $\conv(S)$, is the intersection of all closed half-spaces containing $S$:
$$
\cl(\conv(S)) = \{\pp \in \RR^n : l(\pp) \geq 0 \,\,\forall\,\,l \in \RR[\xx]_1 \textrm{ s.t. } l|_S \geq 0\}.
$$
From a computational point of view, this observation is useless, as the right hand side is hopelessly cumbersome. 
However, if $S$ is the zero set of an ideal $I \subseteq \RR[\xx]$, we can define nice relaxations of the above intersection of infinitely many half-spaces using a classical strengthening of the nonnegativity condition $l|_S \geq 0$. We describe these relaxations in this section. In Section 2.1 we introduce our method for arbitrary ideals in $\RR[\xx]$. In Section 2.2 we specialize to real radical ideals, which occur frequently in applications, and show that in this case, much stronger results hold than for general ideals.

\subsection{General ideals}

Recall that given an ideal $I \subseteq \RR[\xx]$, two polynomials $f$ and $g$ are defined to be congruent modulo $I$, written as $f \equiv g$ mod $I$, if $f-g \in I$. The relation $\equiv$ is an equivalence relation on $\RR[\xx]$ and the equivalence class of a polynomial $f$ is denoted as $f+I$.  The set of all congruence classes of polynomials modulo $I$ is denoted as $\RR[\xx]/I$ and this set is both a ring and a $\RR$-vector space: given $f,g \in \RR[\xx]$ and $\lambda \in \RR$, $(f+I) + (g+I) = (f+g)+I, \,\,\, \lambda (f+I) = \lambda f + I,  \textup{ and } (f+I)(g+I) = fg+I$. Note that if $f \equiv g$ mod $I$, then $f(\ss) = g(\ss)$ for all $\ss \in \V_{\RR}(I)$.

We will say that a polynomial $h\in \RR[\xx]$ is  a sum of squares (\emph{sos}) modulo $I$ if there exist polynomials $g_1,\ldots,g_r \in \RR[\xx]$ such that $h \equiv \sum_{i=1}^r g_i^2$ mod $I$. If $h$ is sos modulo $I$ then we immediately have that 
$h$ is nonnegative on $\V_\RR(I)$. In practice, it
is important to control the degree of the $g_i$ in the sos representation of $h$,  so we will say that $h$ is  $k$-\emph{sos} mod $I$ if $g_1, \ldots, g_r \in \RR[\xx]_k$, where $\RR[\xx]_k$ is the set of polynomials in $\RR[\xx]$ of degree at most $k$. The set of polynomials that are $k$-sos mod $I$, considered as a subset of $\RR[\xx]_{2k}/I$ will be denoted as $\Sigma_k(I)$.

\begin{definition}\label{def:theta_body}
Let $I \subseteq \RR[\xx]$ be a polynomial ideal. We define the $k$-th \emph{theta body} of $I$ to be the set
$$
\TH_k(I) := \{\pp \in \RR^n : l(\pp) \geq 0 \textrm{ for all } l \in \RR[\xx]_1 \textrm{ s.t. } l \textrm{ is } k\textrm{-sos} \textrm{ mod } I\}.
$$
\end{definition}

Since, if $l$ is sos mod $I$ then $l \geq 0$ on $\V_\RR(I)$, and $\TH_k(I)$ is closed,  $\cl(\conv(\V_{\RR}(I))) \subseteq \TH_k(I)$.  Also, $\TH_k(I) \subseteq \TH_{k-1}(I)$, since as $k$ increases, we are potentially intersecting more half-spaces. Thus, the theta bodies of $I$ create a nested sequence of closed convex relaxations of $\conv(\V_\RR(I))$.

We now present a related semidefinite programming relaxation of $\V_\RR(I)$ using the theory of moments.

For $I \subseteq \RR[\xx]$ an ideal, let $\B=\{f_0 + I,f_1+I, \ldots\}$ be a basis for the $\RR$-vector space $\RR[\xx]/I$. 
We assume that the polynomials $f_i$ representing the elements of $\B$ are minimal degree representatives of their equivalence classes $f_i+I$. This makes the set $\B_k := \{f_i+I \in \B: \deg(f_i) \leq k\}$ well-defined. In this paper we will restrict ourselves to a special type of basis $\B$.

\begin{definition}
Let $I \subseteq \RR[\xx]$ be an ideal. A basis $\B=\{f_0+I,f_1+I,\ldots\}$ of $\RR[\xx]/I$ is a $\theta$-\emph{basis} if it satisfies the following conditions:
\begin{enumerate}
\item $f_0(\xx)=1$;
\item $f_i(\xx)=x_i$, for $i=1,\ldots,n$;
\item $f_i$ is monomial for all $i$;
\item if $\deg(f_i),\deg(f_j) \leq k$ then $f_if_j + I$ is in the real span of $\B_{2k}$.
\end{enumerate}
We will also always assume that $\B$ is ordered and that $\deg(f_i) \leq \deg(f_{i+1})$.
\end{definition}

Using Gr\"obner bases theory, one can see that if $\V_{\RR}(I)$ is not contained in any proper affine space, then $\RR[\xx]/I$ always has a $\theta$-basis. For instance, take the $f_i$ in $\B$ to be the \emph{standard monomials} of an \emph{initial ideal} of $I$ with respect to some \emph{total degree monomial ordering} on $\RR[\xx]$  (see for example \cite{CLO}). The methods we describe below work with non monomial bases of $\RR[\xx]/I$ as explained in \cite{GPT}. We restrict to a $\theta$-basis in this survey for ease of exposition and since the main applications we will discuss only need this type of basis.

Fix a $\theta$-basis $\B$ of $\RR[\xx]/I$ and define $\ff_k[\xx]$ to be the column vector formed by all the elements of $\B_k$ in order.
Then $\ff_k[\xx] \ff_k[\xx]^T$ is a square matrix indexed by $\B_k$ with $(i,j)$-entry equal to $f_if_j + I$. By hypothesis, the entries of  $\ff_k[\xx] \ff_k[\xx]^T$ lie in the $\RR$-span of $\B_{2k}$. Let \{ $\lambda_{i,j}^l$ \} be the set of real numbers such that $f_i f_j + I = \sum_{f_l + I \in \B_{2k}} \lambda_{i,j}^l f_l + I$.
We now linearize  $\ff_k[\xx] \ff_k[\xx]^T$ by replacing each element of $\B_{2k}$ by a new variable.

\begin{definition}
Let $I$, $\B$ and \{ $\lambda_{i,j}^l$ \} be as above. Let $\yy$ be a real vector indexed by $\B_{2k}$. The $k$-\emph{th combinatorial moment
matrix} $M_{\B_k}(\yy)$ of $I$  is the real matrix indexed by $\B_k$ whose $(i,j)$-entry is 
$[M_{\B_k}(\yy)]_{i,j}=\sum_{f_l + I \in \B_{2k}} \lambda_{i,j}^l y_l.$
\end{definition}

\begin{example}
Let $I  \subseteq  \RR[x_1,x_2]$ be the ideal $I=\langle x_1^2-2x_2+x_1, x_1x_2 \rangle$. 
Then a $\theta$-basis for $I$ would be $\B=\{1,x_1,x_2,x_2^2,x_2^3,x_2^4, \ldots\}$. 
Let us construct the matrix $M_{\B_1}(\yy)$. Consider the vector $\ff_1[\xx]=(1 \,x_1  \,x_2)^T$, then
$$\ff_1[\xx] \ff_1[\xx]^T = \left( 
\begin{array}{ccc}
1 & x_1 & x_2 \\ x_1 & x_1^2 & x_1x_2 \\ x_2 & x_1x_2 & x_2^2 \end{array} \right)
\equiv  \left( \begin{array}{ccc}
1 & x_1 & x_2 \\ x_1 & 2x_2-x_1 & 0 \\ x_1 & 0 & x_2^2 \end{array} \right) \textup{ mod } I.$$
We now linearize the resulting matrix using $\yy=(y_0,y_1,y_2, \ldots)$, where $y_i$ indexes the $i$th element of $\B$, 
and get
$$M_{\B_1}(\yy)=\left( \begin{array}{ccc}
y_0 & y_1 & y_2 \\ y_1 & 2y_2-y_1 & 0 \\ y_2 & 0 & y_3 \end{array} \right).$$
\end{example}

The matrix $M_{\B_k}(\yy)$ will allow us to define a relaxation of $\cl(\conv(\V_{\RR}(I)))$ that will essentially be the theta body $\TH_k(I)$.

\begin{definition}
Let $I \subseteq \RR[\xx]$ be an ideal and $\B$ a $\theta$-basis of $\RR[\xx]/I$. Then
$$ Q_{\B_k}(I) := \pi_{\RR^n}(\{\yy \in \RR^{\B_{2k}}: y_0=1, M_{\B_k}(\yy) \succeq 0\}) $$
where $\pi_{\RR^n}$ is projection of $\yy \in \RR^{B_{2k}}$ to $(y_1, \ldots, y_n)$,  its coordinates 
indexed by $x_1+I,\ldots,x_n+I$.
\end{definition}

The set $Q_{\B_k}(I)$ is a relaxation of $\conv(\V_\RR(I))$. Pick $\ss \in \V_\RR(I)$ and define $\yy^{\ss} := (f_i(\ss) \,:\, f_i+I \in \B_{2k})$. Then $\yy^{\ss} (\yy^{\ss})^t = M_{\B_k}(\yy^{\ss})$, $y^{\ss}_0 = 1$ and $\pi_{\RR^n}(y^{\ss}) = \ss$. We now show the connection between $Q_{\B_k}(I)$ and $\TH_k(I)$.

\begin{theorem}\label{thm:lasserre_vs_theta_1}
For any ideal $I\subseteq \RR[\xx]$ and any $\theta$-basis $\B$ of $\RR[\xx]/I$, we get 
$\cl(Q_{\B_k}(I)) \subseteq \TH_k(I)$. 
\end{theorem}

\begin{proof}
We start with a general observation concerning $k$-sos polynomials. 
Suppose $h\equiv \sum_{i=1}^r g_i^2 \mod I$ where $g_i \in \RR[\xx]_k$. 
Each $g_i$ can be identified with a real row vector $\hat{g_i}$ such that 
$g_i(\xx) \equiv  \hat{g_i} \ff_k[\xx] \mod I$, and so
 $g_i^2 \equiv \ff_k[\xx]^T \hat{g_i}^T \hat{g_i} \ff_k[\xx]$.
Denoting by $P_h$ the positive semidefinite matrix $\sum_{i=1}^r \hat{g_i}^T \hat{g_i}$ we get 
$h(\xx)\equiv \ff_k[\xx]^T P_h \ff_k[\xx] \mod I$. In general, $P_h$ is not unique. 
Let $\hat{h}$ be the real row vector such that $h(\xx) \equiv \hat{h}\ff_{2k}[\xx]$ mod $I$. 
Then check that for any column vector 
$\yy \in \RR^{\B_{2k}}$, $\hat{h} \yy = P_h \cdot M_{\B_k}(\yy)$, where $\cdot$ stands for
the usual entry-wise inner product of matrices.

Suppose $\pp \in Q_{\B_k}(I)$, and $\yy \in \RR^{\B_{2k}}$ such that 
$y_0=1$, $M_{\B_k}(\yy) \succeq 0$ 
and $\pi_{\RR^n}(\yy)=\pp$. Since $\TH_k(I)$ is closed, 
we just have to show that for any $h = a_0 + \sum_{i=1}^{n} a_i x_i \in \RR[\xx]_1$ 
that is $k$-sos modulo $I$, $h(\pp) \geq 0$. Since $y_0=1$ and $\pi_{\RR^n}(\yy)=\pp$, 
$h(\pp) = a_0 + \sum_{i=1}^{n} a_i p_i = a_0 y_0 + 
\sum_{i=1}^{n} a_i y_i = \hat{h} \yy = P_h \cdot M_{\B_k}(\yy) \geq 0$ since 
$P_h \succeq 0$ and $M_{\B_k}(\yy) \succeq 0$.
\end{proof}

In the next subsection we will see that when $I$ is a real radical ideal, 
$\textup{cl}(Q_{\B_k}(I))$ coincides with $\TH_k(I)$.

The idea of computing approximations of the convex hull of a semialgebraic set 
in $\RR^n$ via the theory of moments and the dual theory of sums of
 squares polynomials is due to Lasserre \cite{Lasserre1,Lasserre2, Lasserre3} 
 and Parrilo \cite{Parrilo:phd,Parrilo:spr}. 
 In his original set up, the moment relaxations obtained are described
 via moment matrices that rely explicitly on the polynomials defining the 
 semialgebraic set. In \cite{Lasserre2}, the focus is on semialgebraic
 subsets of $\{0,1\}^n$ where the equations $x_i^2-x_i$ are used to 
 simplify computations. This idea was generalized in \cite{Laurent} to arbitrary real 
algebraic varieties and studied in detail for zero-dimensional ideals. Laurent showed 
that the moment matrices needed in the approximations of $\conv(\V_\RR(I))$ 
could be computed modulo the ideal defining the variety.  
This greatly reduces the size of the matrices needed, and removes
the dependence of the computation on the specific presentation of the 
ideal in terms of generators. The construction of the set $Q_{\B_k}(I)$ 
is taken from \cite{Laurent}. Since an 
algebraic set is also a semialgebraic set (defined by equations), 
we could apply Lasserre's  method to $\V_\RR(I)$  to get a sequence of approximations 
$\conv(\V_\RR(I))$. The results are essentially the 
same as theta bodies if the generators of $I$ are picked
 carefully. However, by restricting ourselves to real varieties, instead of allowing 
 inequalities, and concentrating on the sum of squares description of theta bodies, as 
 opposed to the moment matrix approach, we can prove some interesting 
 theoretical results that are not covered by the general theory for Lasserre relaxations. Many 
 of the usual results for Lasserre relaxations rely on the existence of a non-empty interior for the 
 semialgebraic set to be convexified which is never the case for a real variety, or compactness of the semialgebraic set 
 which we do not want to impose.

\subsection{Real radical ideals}

Recall from the introduction that given an ideal $I \subseteq \RR[\xx]$ its real radical is the ideal
$$\sqrt[\RR]{I}=\left\{ f \in \RR[\xx] : f^{2m} + \sum g_i^2 \in I, m \in \NN, g_i \in \RR[\xx]\right\}.$$
The importance of this ideal arises from the Real Nullstellensatz.

\begin{theorem}[Bounded degree Real Nullstellensatz \cite{Lombardi}] \label{thm:real_nullstellensatz}
Let $I \subseteq \RR[\xx]$ be an ideal. Then there exists a function 
$\varGamma:\NN\rightarrow \NN$ such that, for all polynomials $f \in \RR[\xx]$ of
degree at most $d$ that vanish on $\V_{\RR}(I)$, $f^{2m} +  \sum g_i^2 \in I$ 
for some polynomials $g_i$ such that
$\deg(g_i)$ and $m$ are all bounded above by $\varGamma(d)$. In particular $\I(\V_{\RR}(I))=\sqrt[\RR]{I}$.
\end{theorem}

When $I$ is a real radical ideal, the sums of squares approach and the moment approach for theta bodies of $I$ coincide, and we get a stronger version of Theorem \ref{thm:lasserre_vs_theta_1}.

\begin{theorem}\label{thm:lasserre_vs_theta}
For any $I\subseteq \RR[\xx]$ real radical and any $\theta$-basis $\B$ of $\RR[\xx]/I$,  $\cl(Q_{\B_k}(I)) = \TH_k(I)$.
\end{theorem}

\begin{proof}
By Theorem \ref{thm:lasserre_vs_theta_1} we just have to show that $\TH_k(I) \subseteq \cl(Q_{\B_k}(I))$. 
By \cite[Prop 2.6]{PowSchei}, the set $\Sigma_k(I)$ of 
elements of $\RR[\xx]_{2k}/I$ that are $k$-sos modulo $I$,
is closed when $I$ is real radical. Let $f \in \RR[\xx]_1$ 
be nonnegative on $\cl(Q_{\B_k}(I))$ and suppose $f+I \not \in \Sigma_k(I)$. 
By the separation theorem, 
we can find $\yy \in \RR^{\B_k}$ such that $\hat{f} \yy < 0$ but $\hat{g} \yy \geq 0$
for all $g + I \in \Sigma_k(I)$, or equivalently, $P_g \cdot M_{\B_k}(\yy) \geq 0$. 
Since $P_g$ runs over all possible positive semidefinite matrices of size $|\B_k|$,
and the cone of positive semidefinite matrices of a fixed size is self-dual, 
we have $M_{\B_k}(\yy) \succeq 0$.
Let $r$ be any real number and consider $g_r  + I := (f+r)^2  + I \in \Sigma_k(I)$. Then 
$\hat{g_r} \yy =  \hat{f^2}  \yy+ 2r \hat{f}  \yy+ r^2 y_0 \geq 0$.
Since $\hat{f} \yy < 0$ and $r$ can be arbitrarily
large, $y_0$ is forced to be positive. So we can scale $\yy$ 
to have $y_0=1$, so that $\pi_{\RR^n}(\yy) \in Q_{\B_k}(I)$. By hypothesis we then have
$f(\pi_{\RR^n}(\yy)) \geq 0$, but by the linearity of $f$, 
$f(\pi_{\RR^n}(\yy))=\hat{f} \yy < 0$ which is a contradiction, so $f$ must be $k$-sos
modulo $I$. This implies that any linear inequality 
valid for $\cl(Q_{\B,k}(I))$ is valid for $\TH_k(I)$, 
which proves $\TH_k(I) \subseteq \cl(Q_{\B,k}(I))$.
\end{proof}

We now have two ways of looking at the relaxations $\TH_k(I)$ --- 
one by a characterization of the linear inequalities that hold on them and the other by a characterization of 
the points in them. These two perspectives complement each other. 
The inequality version is useful to prove (or disprove) convergence of 
theta bodies to $\textup{cl}(\conv(\V_{\RR}(I)))$ while the description via 
semidefinite programming is essential for practical computations. 
All the applications we consider use real radical ideals in which case 
the two descriptions of $\TH_k(I)$ coincide up to closure. 
In some cases, as we now show, the closure can be omitted.

\begin{theorem} \label{thm:omit closure}
Let $I\subseteq \RR[\xx]$ be a real radical ideal and $k$ be a positive integer. 
If there exists some linear polynomial $g$ that is $k$-sos modulo $I$
with a representing matrix $P_g$ that is positive definite, then $Q_{\B_k}(I)$ is 
closed and equals $\TH_k(I)$. 
\end{theorem}

\begin{proof}
For this proof we will use a standard result from convex analysis: 
Let $V$ and $W$ be finite   dimensional vector spaces, $H \subseteq W$ be a cone and 
$A \,:\, V \rightarrow W$ be a linear map such that 
$A(V) \cap \textup{int}(H) \neq \emptyset$. Then $(A^{-1}H)^* = A'(H^*)$ where $A'$ is the adjoint
operator to $A$.  In particular, $A'(H^*)$ is closed in $V'$, the dual vector space to $V$. The proof 
of this result follows, for example, from Corollary~3.3.13 in \cite{BorweinLewis} by setting $K = V$.

Throughout the proof we will identify $\RR[\xx]_l/I$, for all $l$, with 
the space $\RR^{\B_l}$ by simply considering the coordinates in the basis $\B_l$.
Consider the inclusion map $A \,: \, \RR^{\B_1}\cong \RR^{n+1}\rightarrow \RR^{\B_{2k}}$, 
and let $H$ be the cone in $\RR^{\B_{2k}}$ of polynomials
that can be written as a sum of squares of polynomials of 
degree at most $k$. The interior of this cone is precisely the
set of sums of squares $g$ with a positive definite 
representing matrix $P_g$. Our hypothesis then simply states that 
$A(\RR^{\B_1}) \cap \textup{int}(H) \not = \emptyset$ which implies by the 
above result that $A'(H^*)$ is closed. Note that $H^*$ is the set 
of elements $\yy$ in $\RR^{\B_{2k}}$ 
such that $\hat{h} \yy$ is nonnegative for all $h \in H$ and this 
is the same as demanding $P_h \cdot M_{\B_k}(\yy) \geq 0$ 
for all positive semidefinite matrices $P_h$, which is equivalent to 
demanding that $M_{\B_k}(\yy)\succeq 0$. So $A'(H^*)$ is just the set
$\pi_{\RR^{n+1}}(\{ \yy \in \RR^{\B_{2k}}: M_{\B_k}(\yy) \succeq 0\}) $ and by 
intersecting it with the plane $\{\yy \in \RR^{\B_{2k}}: y_0 = 1\}$
we get $Q_{\B_k}(I)$ which is therefore closed.
\end{proof}

One very important case where the conditions of 
Theorem~\ref{thm:omit closure} holds is when 
$I$ is the vanishing ideal of a set of $0/1$ points. This is precisely the case
of most interest in combinatorial optimization.

\begin{corollary}\label{cor:01}
If $S \subseteq \{0,1\}^n$ and $I=\I(S)$, then $Q_{\B_k}(I)=\TH_k(I)$.
\end{corollary}
\begin{proof}
It is enough to show that there is a linear polynomial $g \in \RR[\xx]$ such that $g \equiv
\ff_k[\xx]^T A \ff_k[\xx]$ mod $I$ for a \emph{positive definite} matrix $A$ and some $\theta$-basis $\B$ of $\RR[\xx]/I$. 

Let $A$ be the matrix $$ A =\left(
\begin{array} {cc}
l+1 & \cc^t \\
\cc & D
\end{array}
\right),$$
where $l+1 = | \B_k |$, $\cc \in \RR^{l}$ is the vector with all entries equal to $-2$, and $D \in \RR^{l \times l}$ is the
diagonal matrix with all diagonal entries equal to $4$. This matrix is positive definite since $D$ is positive definite
and its  Schur complement $(l+1) - \cc^t D^{-1} \cc = 1$ is positive.

Since $x_i^2  \equiv x_i$ mod $I$ for $i=1,\ldots,n$ and $\B$ is a monomial basis,
  for any $f+I \in \B$, $f \equiv f^2$ mod $I$. Therefore, the constant (linear polynomial) 
$l+1 \equiv \ff_k[\xx]^T
A\ff_k[\xx]$ mod $I$.
\end{proof}

The assumption that $I$ is real radical seems very strong. 
However, we now establish a 
relationship between the theta bodies of an ideal and 
those of its real radical, showing that 
 $\sqrt[\RR]{I}$ 
determines the asymptotic behavior of $\TH_k(I)$. 
We start by proving a small technical lemma.

\begin{lemma} \label{lemma:epsilon}
Given an ideal $I$ and a polynomial $f$ of degree $d$ such that $-f^{2m}\in \Sigma_k(I)$ 
for some $m,k \in \NN$, the polynomial 
$f + \varepsilon \in \Sigma_{k+4dm}(I)$ for every $\varepsilon > 0$.
\end{lemma}

\pagebreak

\begin{proof}
First, note that for any $l \geq m$ and any $\xi > 0$ we have
$$f^l + \xi = \frac{1}{\xi} \left(\frac{f^l}{2}+\xi \right)^2 + \frac{1}{4 \xi}(-f^{2m})f^{2l-2m}$$
and so $f^l + \xi$ is $(dl+k)$-sos modulo $I$. For $\sigma >0$, 
define the polynomial $p(x)$ to be the truncation of 
the Taylor series of $\sqrt{\sigma^2 + 4 \sigma x}$ at degree $2m-1$ i.e.,
$$p(x)=\sum_{n=0}^{2m-1} (-1)^n \frac{(2n)!}{(n!)^2(1-2n)\sigma^{n-1}}x^n.$$
When we square $p(x)$ we get a polynomial whose terms of degree at most $2m-1$ 
are exactly the first $2m-1$ terms of $\sigma^2 + 4 \sigma x$,
and by checking the signs of each of the coefficients of $p(x)$ we can see 
that the remaining terms of $p(x)^2$ will be negative for odd powers
and positive for even powers. Composing $p$ with $f$ we get
$$(p(f(\xx)))^2 = \sigma^2 + 4 \sigma f(\xx) + \sum_{i=0}^{m-1} a_i f(\xx)^{2m+2i} - \sum_{i=0}^{m-2} b_i f(\xx)^{2m+2i+1}$$
where the $a_i$ and $b_i$ are positive numbers. In particular
$$\sigma^2 + 4 \sigma f(\xx) = p(f(\xx))^2 + \sum_{i=0}^{m-1} a_i f(\xx)^{2i}(-f(\xx)^{2m}) + \sum_{i=0}^{m-2} b_i f(\xx)^{2m+2i+1}.$$
On the right hand side of this equality the only term who is not immediately 
a sum of squares is the last one, but by the above remark,
since $2m+2i+1>m$, by adding an arbitrarily small positive number, it becomes $(d(2m+2i+1)+k)$-sos modulo $I$. 
By checking the degrees throughout the sum, one can see that for 
any $\xi > 0$, $\sigma^2 + 4 \sigma f(\xx) + \xi$ is $(4dm+k)$-sos modulo $I$.
Since $\sigma$ and $\xi$ are arbitrary positive numbers we get the desired result.
\end{proof}

Lemma~\ref{lemma:epsilon}, together with the Real Nullstellensatz, gives us an important relationship 
between the theta body hierarchy of an ideal and that of its real radical.

\begin{theorem}\label{thm:theta_inclusion}
Fix an ideal $I$. Then, there exists a function $\varPsi:\NN\rightarrow \NN$ such that 
$\TH_{\varPsi(k)} \subseteq \TH_k(\sqrt[\RR]{I})$ for all $k$.
\end{theorem}

\begin{proof}
Fix some $k$, and let $f(\xx)$ be a linear polynomial that is 
$k$-sos modulo $\sqrt[\RR]{I}$. This means that there exists some sum of squares
$s(\xx) \in \RR[\xx]_{2k}$ such that $f-s \in \sqrt[\RR]{I}$. 
Therefore, by the Real Nullstellensatz (Theorem~\ref{thm:real_nullstellensatz}), 
$-(f-s)^{2m} \in \Sigma_l(I)$ for $l,m \leq \varGamma(2k)$, where 
$\varGamma:\NN\rightarrow \NN$ depends only on the ideal $I$.
By Lemma~\ref{lemma:epsilon} it follows that $f-s+\varepsilon$ is 
$\varGamma(2k) + 8k \varGamma(2k)$-sos modulo $I$ for every $\varepsilon > 0$. 
Let $\varPsi(k) := \varGamma(2k) + 8k \varGamma(2k)$. 
Then we have that $f+\varepsilon$ is $\varPsi(k)$-sos modulo $I$
for all $\varepsilon > 0$. This means that for every $\varepsilon > 0$, the inequality 
$f+\varepsilon \geq 0$ is valid on $\TH_{\varPsi(k)}(I)$ for $f$ linear and 
$k$-sos modulo $\sqrt[\RR]{I}$. Therefore, $f \geq 0$ is also valid on 
$\TH_{\varPsi(k)}(I)$, and hence, 
$\TH_{\varPsi(k)}(I) \subseteq \TH_{k}(\sqrt[\RR]{I})$.
\end{proof}

Note that the function $\varPsi$ whose existence we just proved, 
can be notoriously bad in practice, as it can be much higher than necessary.
The best theoretical bounds on $\varPsi$ come from quantifier elimination 
and so increase very fast. However, if we are only interested in 
convergence of the theta body sequence, as is often the case, 
Theorem~\ref{thm:theta_inclusion} tells us that we might as well assume that 
our ideals are real radical. 

%%%%%%%%%%%%%%%%%%%%%%%%%%%%%%%%%%%%%%%%%%%%%%%%%%%%%%%%%%%%%%%%%%%%%%%%%%%%%%%%%%%%%%%%%%%%%%%%%%%%%%%%%%%%%%%%%%%%%%%%%%%%%%%%%%

\section{Computing theta bodies}
In this section we illustrate the computation of theta bodies on two examples. In the first example, $\V_{\RR}(I)$ is finite and hence $\conv(\V_\RR(I))$ is a polytope, while in the second example $\V_{\RR}(I)$ is infinite. Convex approximations of polytopes via linear or semidefinite programming have received a great deal of attention in combinatorial optimization where  the typical problem is to maximize a linear function $\cc \cdot \xx$ as $\xx$ varies over the characteristic vectors $\chi^T \in \{0,1\}^n$ of some combinatorial objects $T$. Since this discrete optimization problem is equivalent to the linear program in which one maximizes $\cc \cdot \xx$ over all $\xx \in \conv(\{\chi^T\})$ and $\conv(\{\chi^T\})$ is usually unavailable, one resorts to approximations of this polytope over which one can optimize in a reasonable way. A combinatorial optimization problem that has been tested heavily in this context is the \emph{maximum stable set problem} in a graph which we use as our first example. 
In \cite{ShannonCapacity}, Lov{\'a}sz constructed the \emph{theta body of a graph} which was the first example of a semidefinite programming relaxation of a combinatorial optimization problem. The hierarchy of theta bodies for an arbitrary polynomial ideal are a natural generalization of Lov{\'a}sz's theta body for a graph, which explains their name. Our work on theta bodies was initiated by two problems that were posed by Lov{\'a}sz in \cite[Problems 8.3 and 8.4]{Lovasz} suggesting this generalization and its properties.

\subsection{The maximum stable set problem} \label{sec:stableset}

Let $G=([n],E)$ be an undirected graph with vertex set
$[n]=\{1,\ldots,n\}$ and edge set $E$. A \emph{stable set} in $G$ is a
set $U \subseteq [n]$ such that for all $i,j \in U$, $\{i,j\} \not \in
E$. The maximum stable set problem seeks the stable set of largest
cardinality in $G$, the size of which is the \emph{stability number} of
$G$, denoted as $\alpha(G)$. 

The maximum stable set problem can be modeled as follows. For each
stable set $U \subseteq [n]$, let $\chi^U \in \{0,1\}^n$ be its \emph{characteristic vector} defined as $(\chi^U)_i = 1$ if $i \in U$ and
$(\chi^U)_i = 0$ otherwise.  Let $S_G \subseteq \{0,1\}^n$ be the set
of characteristic vectors of all stable sets in $G$. Then
$\textup{STAB}(G) := \conv(S_G)$ is called the \emph{stable set
  polytope} of $G$ and the maximum stable set problem is, in theory,
the linear program $\textup{max}\{ \sum_{i=1}^{n} x_i \,:\, \xx \in
\textup{STAB}(G) \}$ with optimal value $\alpha(G)$. However,
$\textup{STAB}(G)$ is not known apriori, and so one resorts to
relaxations of it over which one can optimize $\sum_{i=1}^{n} x_i$.

In order to compute theta bodies for this example, we first need to view 
$S_G$ as the real variety of an ideal. The natural ideal to take is, $\I(S_G)$, the vanishing ideal of 
$S_G$. It can be 
checked that this real radical ideal is 
$$I_G := \langle x_i^2 - x_i \,\forall \,\,i \in [n], \,\, x_ix_j \, \forall \,\, \{i,j\} \in E \rangle \subset \RR[x_1, \ldots, x_n].$$
For $U \subseteq [n]$, let $\xx^U := \prod_{i \in U} x_i$. From the
generators of $I_G$ it follows that if $f \in \RR[\xx]$, then $f
\equiv g$ mod $I_G$ where $g$ is in the $\RR$-span of the set of
monomials $\{ \xx^U \,:\, U \textup{ is a stable set in } G \}$. Check
that $$\B := \{\xx^U + I_G \,:\, U \,\,\textup{stable set in} \,\, G\}$$
is a $\theta$-basis of $\RR[\xx]/I_G$ containing $1+I_G, x_1 + I_G, \ldots, x_n
+ I_G$. This implies that $\B_k = \{ \xx^U + I_G \,:\, U \,\,\textup{stable set in} \,\, G, \,\,|U| \leq k \}$, and 
for $\xx^{U_i}+I_G, \xx^{U_j} + I_G \in \B_k$, their product is $\xx^{U_i \cup U_j} + I_G$ which is $0+I_G$ if 
$U_i \cup U_j$ is not a stable set in $G$.  This product formula allows us to compute $M_{\B_k}(\yy)$ 
where we index the element $\xx^U + I_G \in \B_k$ by the set $U$. Since 
$S_G \subseteq \{0,1\}^n$ and $I_G = \I(S_G)$ is real radical, by Corollary~\ref{cor:01}, we have that 

$$\TH_k(I_G) = \left\{ \yy \in \RR^n \,:\,
\begin{array}{l} 
  \exists \, M \succeq 0, \, M \in \RR^{|\B_k| \times |\B_k|}
  \,\textup{such that} \\ 
  M_{\emptyset \emptyset} = 1,\\
  M_{\emptyset \{i\}} = M_{\{i\} \emptyset} = M_{\{i\} \{i\}} = y_i \\
  M_{U U'} = 0 \,\,\textup{if} \,\,U \cup U' \,\, \textup{is not stable
    in}  \,\, G\\
  M_{U U'} = M_{W W'} \,\,\textup{if} \,\, U \cup U' = W \cup W' 
\end{array}
\right \}.$$ 
In particular,  indexing the one element stable sets by the vertices
of $G$,
$$\textup{TH}_1(I_G) = \left\{ \yy \in \RR^n \,:\,
\begin{array}{l} 
\exists \, M \succeq 0, M \in \RR^{(n+1) \times (n+1)}\,\textup{such that} \\ 
M_{00} = 1,\\
M_{0i} = M_{i0} = M_{ii} =  y_i \,\,\forall\,\,i \in [n]\\
M_{ij} = 0 \,\,\forall \,\, \{i,j\} \in E
\end{array}
\right \}.$$

\begin{example}
Let $G =(\{1,2,3,4,5\},E)$ be a pentagon. Then 
$$I_G = \langle x_i^2-x_i \,\,\forall \, i =1, \ldots, 5, \,\, x_1x_2, x_2x_3, x_3x_4, x_4x_5, x_1x_5 \rangle$$
and $$\B = \{1, x_1, x_2, x_3, x_4, x_5, x_1x_3, x_1x_4, x_2x_4, x_2x_5, x_3x_5 \} + I_G.$$
Let $\yy \in \RR^{11}$ be a vector whose coordinates are indexed by the elements of $\B$ in the given order. 
Then 
$$\textup{TH}_1(I_G) = 
\left\{ \yy \in \RR^5 \,:\, \exists \,\, y_6, \ldots, y_{10} \textup{ s.t. }
\left ( \begin{array}{cccccc}
1 & y_1 & y_2 & y_3 & y_4 & y_5 \\
y_1 & y_1 & 0 & y_6 & y_7 & 0 \\
y_2 & 0 & y_2 & 0 & y_8 & y_9 \\
y_3 & y_6 & 0 & y_3 & 0 & y_{10} \\
y_4 & y_7 & y_8 & 0 & y_4 & 0 \\
y_5 & 0 & y_9 & y_{10} & 0 & y_5 
\end{array} \right ) \succeq 0 
\right \}, 
$$
and $\textup{TH}_2(I_G) =$
$$
\left\{ \yy \in \RR^5 \,:\, \exists \,\, y_6, \ldots, y_{10} \textup{ s.t. }
\left ( \begin{array}{ccccccccccc}
1 & y_1 & y_2 & y_3 & y_4 & y_5 & y_6 & y_7 & y_8 & y_9 & y_{10}\\
y_1 & y_1 & 0 & y_6 & y_7 & 0 & y_6 & y_7 & 0 & 0 & 0 \\
y_2 & 0 & y_2 & 0 & y_8 & y_9 & 0 & 0 & y_8 & y_9 & 0 \\
y_3 & y_6 & 0 & y_3 & 0 & y_{10} & y_6  & 0 & 0 & 0 & y_{10}\\
y_4 & y_7 & y_8 & 0 & y_4 & 0 & 0 & y_7 & y_8 & 0 & 0 \\
y_5 & 0 & y_9 & y_{10} & 0 & y_5 & 0 & 0 & 0 & y_9 & y_{10}\\
y_6 &  y_6 & 0 & y_6 & 0 & 0 & y_6 & 0 & 0 & 0 & 0 \\
y_7 &  y_7 & 0 &  0 & y_7 & 0 & 0 & y_7 & 0 & 0 & 0 \\
y_8 &  0 & y_8 &  0 & y_8 & 0 & 0 & 0 & y_8 & 0 & 0 \\
y_9 &  0 & y_9 &  0 & 0 & y_9 & 0 & 0 & 0 & y_9 & 0 \\
y_{10} & 0 & 0 &  y_{10} & 0 & y_{10} & 0 & 0 & 0 & 0 & y_{10}
\end{array} \right ) \succeq 0 
\right \}.
$$
It will follow from Proposition~\ref{prop:odd_cycle} below that $$\textup{STAB}(G) = \TH_2(I_G) \subsetneq \TH_1(I_G).$$
In general, it is a non-trivial task to decide whether two convex bodies coincide and thus to check whether a given theta body, $\TH_k(I)$, equals 
$\overline{\conv(\V_{\RR}(I))}$. One technique is to show that all linear functions $l(\xx)$ such that $l(\xx) \geq 0$ on $\V_{\RR}(I)$ are 
$k$-sos mod $I$. Two illustrations of this method appear shortly.
\end{example}

The first theta body of $I_G$ is exactly the theta body, $\TH(G)$, of $G$ as defined by Lov{\'a}sz 
\cite[Lemma 2.17]{LovaszSchrijver91}. The higher theta bodies 
$\TH_k(I_G)$ shown above give a nested sequence of convex relaxations 
of $\textup{STAB}(G)$ extending Lov{\'a}sz's $\TH(G)$. 
The problem $$\textup{max}\{ \sum_{i=1}^{n} x_i \,:\, \xx \in
\textup{TH}(G) \}$$ can be solved to arbitrary precision
in polynomial time in the size of $G$ via semidefinite programming. 
The optimal value of this semidefinite program
is called the \emph{theta number} of $G$ and provides an upper
bound on $\alpha(G)$. See \cite[Chap. 9]{GLS} and
\cite{BruceShepherd} for more on the stable set problem and
$\textup{TH}(G)$.  The body $\textup{TH}(G)$ was the first example of
a semidefinite programming relaxation of a discrete optimization problem and snowballed the
use of semidefinite programming in combinatorial optimization. See \cite{LaurentRendl,
  Lovasz} for surveys. 
  
Recall that a graph $G$ is \emph{perfect} if and
only if $G$ has no induced odd cycles of length at least five, or their
complements. Lov{\'a}sz showed that $\textup{STAB}(G) =
\textup{TH}(G)$ if and only if $G$ is perfect. This equality shows
that the maximum stable set problem can be solved in polynomial time
in the size of $G$ when $G$ is a perfect graph, and this geometric
proof is the only one known for this complexity result. 
Since a pentagon is not 
perfect, it follows that for $G$ a pentagon, $\textup{STAB}(G) \subsetneq \TH_1(I_G)$. 
We will see in Corollary~\ref{cor:finite_real_variety} that when 
$\V_{\RR}(I)$ is finite, 
then there is some finite $k$ for which 
$\TH_k(I) = \conv(\V_\RR(I))$. Since 
no monomial in the basis $\B$ of $\RR[\xx]/I_G$ has degree
larger than $\alpha(G)$, for any $G$, 
$\textup{STAB}(G) = \textup{TH}_{\alpha(G)}(I_G)$ which proves that 
when $G$ is a pentagon, $\textup{STAB}(G) = \textup{TH}_{2}(I_G)$.
We now prove a second, general reason why $\textup{STAB}(G) = \textup{TH}_{2}(I_G)$
when $G$ is a pentagon.

A simple class of linear inequalities that are valid on $\textup{STAB}(G)$ are the \emph{odd cycle inequalties}, 
$\sum_{i \in C} x_i \leq \frac{|C|-1}{2}$, where $C$ is an odd cycle in $G$.
 
\begin{proposition} \label{prop:odd_cycle}
For any graph $G$, all odd cycle inequalities are valid on $\TH_2(I_G)$. 
\end{proposition}

\begin{proof} 
  Let $n = 2k+1$ and $C$ be an $n$-cycle. Then $I_C = \langle
  x_i^2-x_i, \,\, x_ix_{i+1} \,\,\forall \,\, i \in [n] \rangle$ where
  $x_{n+1} = x_1$.  Therefore, $(1-x_i)^2 \equiv 1-x_i$ and
  $(1-x_i-x_{i+1})^2 \equiv 1 -x_i-x_{i+1} \,\,\textup{mod}\,\,I_C$.
  This implies that, mod $I_C$,
$$p_i^2:=((1-x_1)(1-x_{2i}-x_{2i+1}))^2 \equiv p_i = 1 - x_1
-x_{2i}-x_{2i+1} + x_1x_{2i} + x_1x_{2i+1}.$$
Summing over $i=1,..,k$, we get 
$$\sum_{i=1}^{k} p_i^2 \equiv k - kx_1 -\sum_{i=2}^{2k+1}x_i +
\sum_{i=3}^{2k} x_1x_i \,\,\textup{mod}\,\,I_C$$
since $x_1x_2$ and $x_1x_{2k+1}$ lie in 
$I_C$.  Define $g_i := x_1(1-x_{2i+1}-x_{2i+2})$. Then $g_i^2 - g_i
\in I_C$ and mod $I_C$ we get that
$$\sum_{i=1}^{k-1} g_i^2 \equiv (k-1)x_1 - \sum_{i=3}^{2k} x_1x_i,
\,\,\textup{which implies} \,\, \sum_{i=1}^{k} p_i^2 +
\sum_{i=1}^{k-1} g_i^2 \equiv k - \sum_{i=1}^{2k+1}x_i.$$ 
Thus, $k - \sum_{i=1}^{2k+1}x_i$ is $2$-sos mod $I_C$. 

Now let $G$ be any graph and $C$ be an induced $(2k+1)$-cycle in $G$. 
Then since $I_C \subseteq I_G$, $k - \sum_{i=1}^{2k+1}x_i$ is $2$-sos mod $I_G$ 
which proves the result.
\end{proof}

\begin{corollary} \label{cor:odd_cycle_theta2}
If $G$ is an odd cycle of length at least five, $\textup{STAB}(G) = \TH_2(I_G)$. 
\end{corollary}

\begin{proof}
It is known that the facet inequalities of $\textup{STAB}(G)$ for $G$ an $n=2k+1$-cycle are:
$$x_i \geq 0 \,\,\forall \,\,i\in [n], \,\,\, 1 - x_i - x_{i+1}  \geq 0 \,\,\forall \,\,i \in [n],
\,\,\,\textup{ and } \,\,\, k - \sum_{i=1}^{n} x_i \geq 0$$
(see for instance,  \cite[Corollary 65.12a]{SchrijverB}). Clearly, $x_i$ is $2$-sos mod $I_G$ for all $i \in [n]$, and 
check that $1-x_i-x_j \equiv (1-x_i-x_j)^2$ mod $I_G$ for all $\{i,j\} \in E$. By Proposition~\ref{prop:odd_cycle}, 
$k - \sum_{i=1}^{n} x_i$ is 2-sos mod $I_G$. 
If $l(\xx) \in \RR[\xx]_1$ such that $l(\xx) \geq 0$ on $\textup{STAB}(G)$, then $l(\xx)$ is a nonnegative linear 
combination of the facet inequalities of $\textup{STAB}(G)$, and hence $l(\xx)$ is 2-sos mod $I_G$. 
\end{proof}

Using the same method as in Proposition~\ref{prop:odd_cycle}, one can show that the 
\emph{odd antihole} and \emph{odd wheel} inequalities \cite[Chap. 9]{GLS} 
that are valid for $\textup{STAB}(G)$ are
also valid for $\textup{TH}_2(I_G)$. Schoenebeck \cite{Schoenebeck}
has shown that there is no constant $k$ such that
$\textup{STAB}(G) = \textup{TH}_k(I_G)$ for all graphs $G$ (as
expected, unless P=NP).  However, no explicit family of graphs that
exhibit this behavior is known.

\subsection{An infinite variety} Consider the cardioid with defining equation $$(x_1^2 + x_2^2 + 2x_1)^2 = 4(x_1^2+x_2^2).$$
This plane curve is the real variety of the ideal $I = \langle h \rangle$ where $$h = x_1^4 + 2x_1^2x_2^2 + x_2^4 + 4x_1^3 + 4x_1x_2^2 - 4x_2^2.$$
Points on this variety are parametrized by the angle $\theta$ made by the line segment from the origin to the point, with the $x_1$-axis, by the equations: $$x_1(\theta) = 2 cos \, \theta \, (1 - cos \, \theta), \,\,\,x_2(\theta) = 2 sin \, \theta \, (1 - cos \, \theta).$$
It can be checked that the half space containing the origin defined by the tangent to the cardiod at $(x_1(\theta), x_2(\theta))$ is 
$$x_1 \frac{sin \, 2 \theta + sin \, \theta}{2 cos^2 \, \theta + cos \, \theta - 1} - x_2 + \frac{sin \, 2 \theta - 2 sin \, \theta}{2 cos^2 \, \theta + cos \, \theta - 1} \geq 0.$$

We now compute theta bodies of this ideal. Since the defining ideal of the cardiod is generated by a quartic polynomial, no linear polynomial is a sum of squares of linear polynomials modulo this ideal. Therefore, $\TH_1(I) = \RR^2$. A $\RR$-vector space basis $\B$ of $\RR[x_1,x_2]/I$ is given by the monomials in $x_1$ and $x_2$ that are not divisible by $x_1^4$. In particular, 
$$\B_4 = \{ 1, x_1, x_2, x_1^2, x_1x_2, x_2^2, x_1^3, x_1^2x_2, x_1x_2^2, x_2^3, x_1^3x_2, x_1^2x_2^2, x_1x_2^3, x_2^4 \} + I,$$ and suppose the coordinates of $\yy \in \RR^{\B_4}$ are indexed as follows ($y_0 = 1$):
$$\begin{array}{c|c|c|c|c|c|c|c|c|c|c|c|c|c}
1 &  x_1 &  x_2 &  x_1^2 &  x_1x_2 &  x_2^2 &  x_1^3 &  x_1^2x_2 &  x_1x_2^2 &  x_2^3 & x_1^3x_2 &  x_1^2x_2^2 &  x_1x_2^3 &  x_2^4 \\
\hline
 1 & y_1 & y_2 & y_3 & y_4 & y_5 & y_6 & y_7 & y_8 & y_9 & y_{10} & y_{11} & y_{12} & y_{13}
\end{array}$$
To compute $M_{\B_2}(\yy)$ we need to express $f_i f_j + I$ as a linear combination of elements in $\B_4$ for every $f_i+I, f_j + I \in \B_2$  and then linearize this linear combination using $\yy$. For example, 
$x_1^4 + I = -2x_1^2x_2^2 - x_2^4 - 4x_1^3 - 4x_1x_2^2 + 4x_2^2 + I$ which linearizes to 
$$T := -2y_{11} - y_{13} - 4y_6 - 4y_8 + 4y_5.$$ Doing all these computations shows that $$\TH_2(I)=
\left \{ (y_1,y_2)\,:\, \exists \yy \in \RR^{13} \textup{ s.t. }
\left ( \begin{array}{cccccc}
1 & y_1 & y_2 & y_3 & y_4 & y_5 \\
y_1 & y_3 & y_4 & y_6 & y_7 & y_8 \\
y_2 & y_4 & y_5 & y_7 & y_8 & y_9 \\
y_3 & y_6 & y_7 & T & y_{10} & y_{11} \\
y_4 & y_7 & y_8 & y_{10} & y_{11} & y_{12} \\
y_5 & y_8 & y_9 & y_{11} & y_{12} & y_{13} 
\end{array} \right ) \succeq 0
\right \} 
$$
For a given $\theta$ we can find the maximum $t\in \RR$ such 
that $(t \cos(\theta),t \sin(\theta))$ is in $\TH_2(I)$, using an SDP-solver.
This point will be on the boundary of $\TH_2(I)$, and if we vary $\theta$, 
we will trace that boundary. In Fig. \ref{fig:cardioid} we show
the result obtained by repeating this procedure over $720$ equally spaced 
directions and solving this problem using SeDuMi 1.1R3 \cite{sedumi}.

\begin{figure}[ht]
\begin{center}
\hfill 
\includegraphics[scale=0.3]{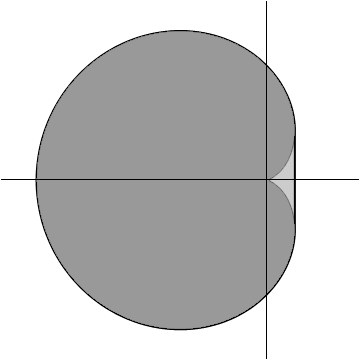}\hfill\
\caption{\small $\TH_2(I)$ compared with the cardioid.} \label{fig:cardioid}
\end{center}
\end{figure}

From the figure, $\TH_2(I)$ seems to match the convex hull of the cardiod, 
suggesting that $\TH_2(I)=\conv(\V_{\RR}(I))$. To prove this equality, we 
would have to construct a sos representation modulo the ideal for each of the tangents to the cardiod.
Independently of that, it is clear from the figure that $\TH_2(I)$ does a very good job of approximating 
this convex hull.

We now explain how one can also optimize a linear function over a theta body using the original sos definition of the body.
For a vector $\cc \in \RR^n$, maximizing $\cc \cdot \xx$ over $\TH_k(I)$ is the same as minimizing $\lambda \in \RR$
such that $\lambda - \cc \cdot \xx$ is nonnegative on $\TH_k(I)$. Under some mild assumptions such as 
compactness of $\TH_k(I)$ or $I$ being real radical, this is precisely the same as minimizing $\lambda \in \RR$
such that $\lambda - \cc \cdot \xx$ is $k$-sos modulo $I$.

Consider the previous example of the cardioid and let $\cc=(1,1)$. If we want to optimize in that direction over $\TH_2(I)$,
we need to minimize $\lambda$ such that $\ell_{\lambda}(x_1,x_2):=\lambda - x_1 - x_2$ is $2$-sos modulo $I$. 
If $\ff_2(\xx) := (1 \,\,x_1 \,\,x_2 \,\,x_1^2 \,\,x_1x_2 \,\,x_2^2)^T$,
then $\ell_{\lambda}(x_1,x_2)$ is $2$-sos modulo $I$ if and only if there exists a $6$ by $6$ positive semidefinite matrix $A$ such that
$$\ell_{\lambda}(x_1,x_2) \equiv \ff_2(\xx)^T A \ff_2(\xx) \,\,\textup{ mod } I.$$
By carrying out the multiplications and doing the simplifications using the ideal one gets that this is equivalent to finding
$A \succeq 0$ such that 
$$\begin{array}{rl}
\ell_{\lambda}(x_1,x_2) = & A_{11}+2A_{12}x_1 + 2A_{13}x_2 + (2A_{14}+A_{22})x_1^2 \\ 
& \,\,\,+ (2A_{15}+2A_{23})x_1x_2 +(2A_{16}+A_{33}+4A_{44})x_2^2 \\
& \,\,\,+ (2A_{24}-4A_{44})x_1^3 + (2A_{25}+2A_{34})x_1^2x_2  \\
& \,\,\,+ (2A_{26}+2A{35}-4A_{44})x_1x_2^2 + 2A_{36}x_2^3 + 2A_{45}x_1^3x_2 \\
& \,\,\,+ (2A_{46}+A_{55}-2A_{44})x_1^2x_2^2 + 2A_{56}x_1x_2^3 + (A_{66}-A_{44})x_2^4
\end{array}$$
so our problem is minimizing $\lambda$ such that there exists $A \succeq 0$ verifying
$$
\begin{array}{l}
A_{11}=\lambda;\\
A_{12}=A_{13}=1/2;\\
2A_{14}+A_{22}=2A_{15}+2A_{23}=2A_{16}+A_{33}+4A_{44}=2A_{24}-4A_{44}\\
\ \ =2A_{25}+2A_{34}=2A_{26}+2A{35}-4A_{44}=A_{36}=A_{45}\\
\ \ =2A_{46}+A_{55}-2A_{44}=A_{56}=A_{66}-A_{44}=0
\end{array}
$$

\begin{figure}[ht]
\begin{center}
\hfill 
\includegraphics[scale=0.4]{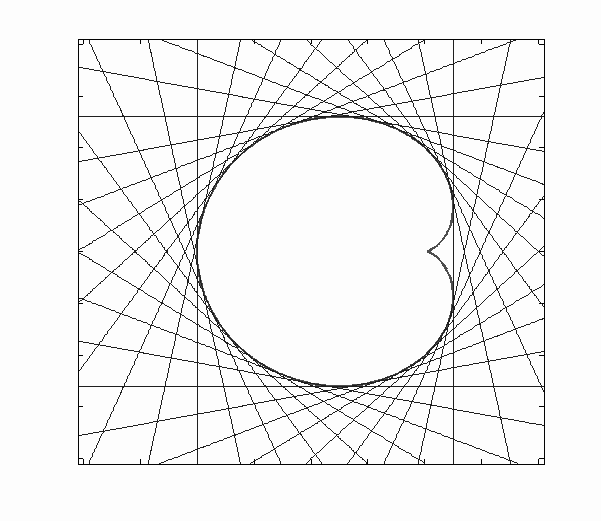}\hfill\
\caption{\small Outer contour of $\TH_2(I)$.} \label{fig:cardioid2}
\end{center}
\end{figure}

By taking $\cc = (\cos(\theta),\sin(\theta))$, varying $\theta$ and tracing all the optimal $\ell_{\lambda}$ obtained one can
get a contour of $\TH_2(I)$. In Fig.~\ref{fig:cardioid2} we can see all the $l_\lambda$'s obtained by repeating this process $32$ times
using SeDuMi.

%%%%%%%%%%%%%%%%%%%%%%%%%%%%%%%%%%%%%%%%%%%%%%%%%%%%%%%%%%%%%%%%%%%%%%%%%%%%%%%%%%%%%%%%%%%%%%%%%%%%%%%%%%%%%%%%%%

\section{Convergence and Exactness}

In any hierarchy of relaxations, one crucial question is 
that of convergence: under what conditions does the sequence
$\TH_k(I)$ approach $\cl(\conv(\V_{\RR}(I)))$? 
Another important question is whether the 
limit is actually reached in a finite number of steps. In this section we will 
give a brief overview of the known answers to these questions. 

In general, convergence is not assured. 
In fact, for the ideal $I=\left< x^3-y^2 \right>$, 
whose real variety is a cusp, the 
closure of $\conv(\V_{\RR}(I))$ is the half plane defined by $x \geq 0$. 
However, all the theta bodies of $I$ are just 
$\RR^2$. This follows immediately from 
the fact that no linear polynomial can ever be written as a 
sum of squares modulo this ideal.
Fortunately, many interesting ideals behave 
nicely with respect to the theta body hierarchy. 
In particular, the central result in this area 
tells us that compactness of $\V_\RR(I)$ implies 
convergence of the theta body hierarchy of $I$.

\begin{theorem}
For any ideal $I \subseteq \RR[\xx]$ such that $\V_{\RR}(I)$ is compact, 
$$\bigcap_{k=0}^{\infty}\TH_k(I)=\cl(\conv(\V_{\RR}(I))).$$
\end{theorem}

\begin{proof}
Let $f_1,\ldots,f_k$ be a set of generators for $I$. 
We can think of $\V_{\RR}(I)$ as the compact semialgebraic set 
$S=\{\xx \in \RR^n: \pm f_i(\xx) \geq 0\}$. Then Schm{\"u}dgen's Positivstellensatz \cite{SCHM} applied to $S$, 
guarantees that any linear polynomial $l(\xx)$ that is strictly positive on $\V_{\RR}(I)$ has a sos representation 
modulo $I$. Hence, $l(\xx) \geq 0$ is valid for $\TH_k(I)$ for some $k$.
\end{proof}

In general, however, we are interested in finite convergence more than in asymptotic convergence, 
since in that case the theta bodies can give us
a representation of the closure of the convex hull of the variety as 
the projection of a spectrahedron, an important theoretical problem on its own. 
We will say that an ideal $I$ is {\bf $\TH_k$-exact} if 
$\TH_k(I)=\cl(\conv(\V_\RR(I)))$ and {\bf $\TH$-exact} if it is $\TH_k$-exact for some $k$. 
We first note that to study $\TH$-exactness of an ideal depends only on the real radical of the ideal. 

\begin{theorem}\label{thm:real_radical_exactness}
An ideal $I \subseteq \RR[\xx]$ is $\TH$-exact if and only if $\sqrt[\RR]{I}$ is $\TH$-exact.
\end{theorem}

\begin{proof}
The ``if'' direction follows from Theorem~\ref{thm:theta_inclusion}, 
while the ``only if'' direction follows from the fact that 
$I \subseteq \sqrt[\RR]{I}$.
\end{proof}

An important case in which $\TH$-exactness holds is when $\V_\RR(I)$ is finite.

\begin{corollary} \label{cor:finite_real_variety}
If $I$ is an ideal such that $\V_{\RR}(I)$ is finite, then $I$ is $\TH$-exact. 
\end{corollary}

\begin{proof}
By Theorem \ref{thm:real_radical_exactness} we can assume that $I$ is real radical. 
If $\V_{\RR}(I)= \{\pp_1,\ldots,\pp_m\}$, then we can construct
interpolators $g_1,\ldots,g_m \in \RR[\xx]$ such that 
$g_i(\pp_j) = 1$ if $j = i$ and $0$ otherwise. Assume that $k$ is the highest 
degree of a $g_i$. Since $g_i -g_i^2$ vanishes on $\V_{\RR}(I)$,  
$g_i-g_i^2 \in \sqrt[\RR]{I} = I$ and therefore, $g_i$ is $k$-sos mod $I$ for $i=1,\ldots,m$.

If $f(\xx) \geq 0$ on $\V_\RR(I)$, then check that $f(\xx) \equiv \sum_{i=1}^{m} f(p_i)g_i(\xx)$ mod $I$. Since 
the polynomial on the right hand side is a nonnegative combination of the $g_i$'s which are $k$-sos mod $I$, 
$f$ is $k$-sos mod $I$. In particular,  all the linear polynomials that are nonnegative on $\V_{\RR}(I)$ are $k$-sos 
mod $I$ and so, $\TH_k(I)= \conv(\V_{\RR}(I))$.
\end{proof}

Under the stronger assumption that $\V_{\CC}(I)$ is finite, 
$\TH$-exactness of $I$ follows from a result of 
Parrilo (see \cite{LaurentSosSurvey} for a proof).
The work in \cite{LaLauRos} implies that when $\V_{\RR}(I)$ is finite, 
the bodies $Q_{\B_k}(I)$ converges in a finite number of steps, 
a slightly weaker result than the one in Corollary~\ref{cor:finite_real_variety}. 
Finite varieties are of great importance in practice. For instance, 
they are precisely the feasible regions of the $0/1$ 
integer programs. 

To study $\TH$-exactness in more detail, a better characterization is needed. 

\begin{theorem}\label{thm:theta_sos}
A real radical ideal $I \subseteq \RR[\xx]$ is $\TH_k$-exact 
if and only if all linear polynomials $l$ that are non-negative in $\V_{\RR}(I)$
are $k$-sos modulo $I$.
\end{theorem}

\begin{proof}
The ``if'' direction follows from the definition of theta bodies. The ``only if'' direction 
is a consequence of the proof of Theorem \ref{thm:lasserre_vs_theta} which said that 
when $I$ is real radical, $\TH_t(I) = \textup{cl}(Q_{\B_t}(I))$ for all $t$.
Recall that in the second part of that proof we showed that any linear inequality valid for 
$\cl(Q_{\B_t}(I))$ was $t$-sos mod $I$. Therefore, if $\TH_k(I) = \textup{cl}(\conv(\V_\RR(I)))$, 
then being nonnegative over 
$\cl(Q_{\B_k}(I)) = \TH_k(I)$ is equivalent to being nonnegative over $\V_{\RR}(I)$ 
giving us the intended result.
\end{proof}

The last general result on exactness that we will present, taken 
from \cite{GouNet}, is a negative one. 
Given a point $\pp \in \V_{\RR}(I)$ we define the \emph{tangent space}
of $\pp$, $T_{\pp}(I)$, to be the affine space passing through 
$\pp$ and orthogonal to the vector space spanned by the gradients of all polynomials
that vanish on $\V_{\RR}(I)$. We say that a point 
$\pp$ in $\V_{\RR}(I)$ is \emph{convex-singular}, 
if it is on the boundary of $\conv(\V_{\RR}(I))$ and $T_{\pp}(I)$ is not 
tangent to  $\conv(\V_{\RR}(I))$ i.e., $T_\pp(I)$ intersects
the relative interior of $\conv(\V_{\RR}(I))$.

\begin{theorem}\label{thm:convex_singularity}
An ideal with a convex-singularity is not $\TH$-exact.
\end{theorem}

\begin{proof}
Let $I$ be an ideal and $\pp$ a convex-singular
point of $\V_{\RR}(I)$, and $J=\I(\V_{\RR}(I))$.
By Theorem~\ref{thm:real_radical_exactness}, 
it is enough to show that $J$ is not $\TH$-exact. 
Let $l(\xx)$ be a linear polynomial that is positive on the relative 
interior of $\conv(\V_{\RR}(I))$ and zero at $\pp$, which we can 
always find since $\pp$ is on the boundary of $\conv(\V_{\RR}(I))$. 
If $J$ was $\TH_k$-exact for some $k$, then since $J$ 
is real radical, by Theorem~\ref{thm:theta_sos} we would be able to write 
$l(\xx)=\sigma(\xx)+g(\xx)$ where $\sigma$ is a 
sum of squares and $g \in J$. Evaluating at $\pp$,
we see that $\sigma(\pp) = 0$,  which 
since $\sigma$ is a sum of squares, implies that $\nabla \sigma(\pp)=0$. 
Therefore, we must have $\nabla l = \nabla g (\pp)$. 
Let $\qq$ be a point in the relative interior of 
$\conv(\V_{\RR}(I))$ that is also in $T_{\pp}(I)$. Then by the definition of $T_\pp(I)$,
$$l(\qq)=(\qq-\pp)\nabla l = (\qq-\pp)\nabla g (\pp)=0$$
which contradicts our choice of $l$. Hence, $I$ is not $\TH_k$-exact.
\end{proof}

Even in the presence of convex-singular points, the theta bodies of the defining ideal can 
approximate the convex hull of the real variety arbitrarily well, but 
they will not converge in finitely many steps. An example of this is the ideal 
$I=\left<x^4-x^3+y^2\right>$. It has a compact real variety with
a convex-singularity at the origin. Hence we have asymptotic convergence of the theta bodies, 
but not finite convergence. The first theta body of this ideal is $\RR^2$ and in 
Fig.~\ref{cusp} we can
see that the next two theta bodies already closely approximate the convex hull of the variety.
\begin{figure}[ht]
\begin{center}
\hfill 
\includegraphics[scale=0.3]{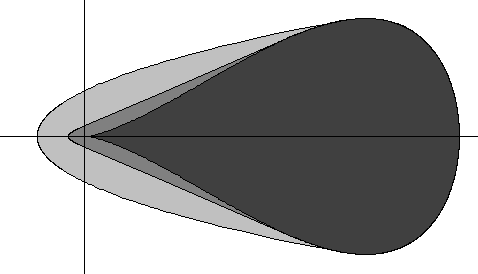}\hfill\
\caption{\small The second and third theta bodies of $I=\left<x^4-x^3+y^2\right>$.} \label{cusp}
\end{center}
\end{figure}

In the remainder of this section we focus 
on a more restricted exactness question. 
Problem 8.3 in \cite{Lovasz} motivates the question of which ideals in $\RR[\xx]$ 
are $\TH_1$-exact. Of particular interest are vanishing ideals of finite sets of points in $\RR^n$ and 
ideals arising in combinatorial optimization.

\begin{theorem}\label{thm:exact_points}
Let $S$ be a finite set of points in $\RR^n$, then 
$I=\I(S)$ is $\TH_1$-exact if and only if for each facet $F$ of the polytope $\conv(S)$
there exists a hyperplane $H$ parallel to $F$ such that $S \subseteq F \cup H$.
\end{theorem}

\begin{proof}
We can assume without loss of generality that $\conv(S)$ is 
full-dimensional, otherwise we could restrict ourselves to its affine hull.
Assume $I$ is $\TH_1$-exact and let $F$ be a facet of $\conv(S)$. 
We can find a linear polynomial $l$ such that $l(\pp)=0$ for all $\pp \in F$
and $l(\pp)\geq 0$ for all $\pp \in S$. By Theorem~\ref{thm:theta_sos},
$l$ must be $1$-sos modulo $I$ which implies 
$l\equiv \sum_{i=1}^r g_i^2 \mod I$ for some linear polynomials $g_i$. 
In particular, $l(\pp)= \sum_{i=1}^r (g_i(\pp))^2$ for all $\pp \in S$, 
and since $l$ vanishes in all points of $S \cap F$, so must the $g_i$. 
This implies that all the $g_i$'s vanish on the 
hyperplane $\{\xx \in \RR^n: l(\xx)=0\}$, since they are linear. This is 
equivalent to saying that every $g_i$ is a scalar multiple of $l$,
 and therefore, $l\equiv\lambda l^2 \mod I$
for some nonnegative $\lambda$. Then $l-\lambda l^2 \in I$ implies that 
$S \subseteq \{\xx:l(\xx)-\lambda l(\xx)^2=0\}$ which is the 
union of the hyperplanes $\{\xx:l(\xx)=0\}$ and $\{\xx:l(\xx)=\sqrt{1/\lambda}\}$. 
So take $H$ to be the second hyperplane.

Suppose now for each facet $F$ of $\conv(S)$ there 
exists a hyperplane $H$ parallel to $F$ such that $S \subseteq F \cup H$.
This implies that for any facet $F$ and a fixed linear inequality $l_F(\xx)\geq 0$ that is
valid on $S$ and holds at equality on $F$, $l_F$ attains the same 
nonzero value at all points in $S \setminus F$. 
By scaling we can assume that value to be one,
which implies that $l_F(1-l_F)$ is zero at all points in $S$, and hence,
$l_F(1-l_F) \in I$. But then $l_F = l_F^2 + l_F(1-l_F)$ is $1$-sos mod $I$.
By Farkas Lemma, any valid linear inequality for a 
polytope can be written as a nonnegative linear combination of the facet inequalities of the polytope.
Therefore, any linear polynomial $l$ that is nonnegative over $S$ is 
a positive sum of $1$-sos polynomials mod $I$,  and hence, $l$ is $1$-sos mod $I$. 
Now using Theorem \ref{thm:theta_sos}, $I$ is $\TH_1$-exact.
\end{proof}

We call a polytope $P$ a $k$-\emph{level} polytope if for any facet 
$F$ of $P$ and any supporting hyperplane $H$ of $F$, there are $k-1$ hyperplanes
$H_1,...,H_{k-1}$ all parallel to $H$ such that the vertices of $P$ 
are contains in $H \cup H_1 \cup \dots \cup H_{k-1}$. 
Theorem \ref{thm:exact_points} states that $\I(S)$ is $\TH_1$-exact 
if and only if $S$ is the set of vertices of a $2$-level polytope. Polytopes with integer 
vertices that are $2$-level are called \emph{compressed polytopes} 
in the literature \cite{StanleyCompressed, Sullivant}. 

\begin{figure}[ht]
\begin{center}
\hfill 
\includegraphics[scale=0.2]{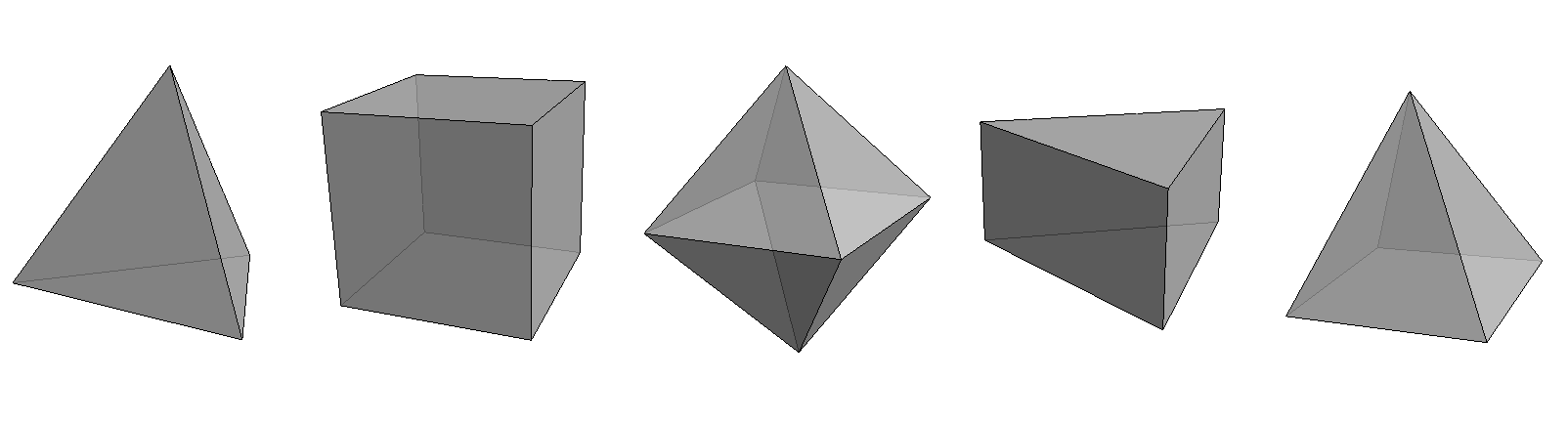}\hfill\
\caption{\small The five $2$-level polytopes in $\RR^3$ up to affine transformations.} \label{Fig:good_poly}
\end{center}
\end{figure}

Examples of $2$-level polytopes include simplices, cubes and cross-polytopes. 
In $\RR^3$, up to affine transformations, there are only five different 
$2$-level polytopes  and they are shown in Fig.~\ref{Fig:good_poly}. 
An example of a polytope that is not $2$-level is the truncated cube shown
in Figure \ref{Fig:bad_poly}. Three parallel translates of the hyperplane spanned by the slanted face 
are needed to contain all vertices of the polytope, and hence this is a $3$-level polytope.
Combinatorial properties of $2$-level polytopes can be found in \cite{GPT}.

\begin{figure}[ht]
\begin{center}
\hfill 
\includegraphics[scale=0.2]{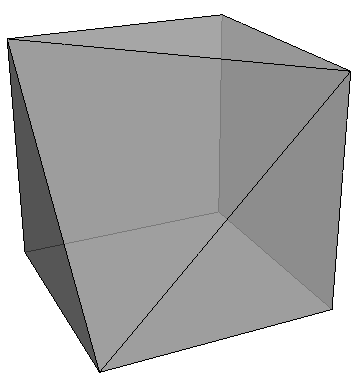}\hfill\
\caption{\small This polytope is $3$-level but not $2$-level.} \label{Fig:bad_poly}
\end{center}
\end{figure}

The result in Theorem \ref{thm:exact_points} can be generalized to a (weak) sufficient criterion for $\TH_k$-exactness.

\begin{theorem} \label{thm:level_polytope}
If $S \subseteq \RR^n$ is the set of vertices of a $(k+1)$-level polytope then $\I(S)$ is $\TH_k$-exact.
\end{theorem}

\begin{proof}
As before, we may assume that $\conv(S)$ is full-dimensional, 
and as observed in the proof of Theorem \ref{thm:exact_points},
it is enough to prove that for every facet $F$ of $\conv(S)$, if $h_F$ 
is a linear polynomial that is zero on $F$ and non-negative on $S$
then $h_F$ is $k$-sos modulo $\I(S)$. We can also assume that $F$ 
is contained in the hyperplane $\{\xx \in \RR^n:x_1=0\}$, and that
all points in $S$ have nonnegative first coordinate, as all the 
properties we are interested in are invariant under translations.

Then, by scaling, we may assume $h_F(\xx) = x_1$, and the 
$(k+1)$-level property tells us that there exist $k$ positive values $a_1,\dots,a_k$ 
such that all points in $S$ have the first coordinate in the set 
$\{0,a_1,\dots,a_k\}$. Then we can construct a one variable 
Lagrange interpolator $g$ of degree $k$ such that $g(0)=0$ and $g(a_i)=\sqrt{a_i}$  for $i=1,\dots,k$. 
This will imply that $h_F(\xx)\equiv g(x_1)^2$ modulo $\I(S)$ and 
so $h_F$ is $k$-sos modulo $\I(S)$.
\end{proof}

This sufficient condition is very restrictive in general, and in fact can be arbitrarily bad.

\begin{example}
We saw in Corollary~\ref{cor:odd_cycle_theta2} that 
if $G$ is a $(2k+1)$-cycle then the ideal $I_G$ is $\TH_2$-exact. 
However notice that we need $k+1$ parallel translates of the 
hyperplane $k = \sum x_i$ to cover all the vertices of $\textup{STAB}(G)$, since the incidence vectors of 
the stable sets in $G$ can 
have anywhere from $0$ to $k$ entries equal to one.
Theorem~\ref{thm:level_polytope} would only guarantee that $I_G$ is $\TH_k$-exact.
\end{example}

When $\V_\RR(I)$ is not finite, $\TH_1$-exactness becomes harder to guarantee. 
A useful result in this direction is an alternative
characterization of the first theta body of any ideal $I$. 
Given any ideal $I$, if we take all convex quadratic polynomials in $I$, 
and intersect the convex hulls of their 
zero sets, we obtain $\TH_1(I)$ exactly. 
This result, proved in \cite[Theorem 5.4]{GPT}, can be used in some simple cases
to both prove or disprove $\TH_1$-exactness.

Many other questions on convergence remain open. For example, 
there is no known example of a smooth variety whose theta body hierarchy does not 
converge finitely, but there is also no reason to believe that such an example does not exist. 
In all studied examples of smooth hypersurfaces, 
convergence happens at the first non-trivial theta body. 
Either an explanation for this behavior, or a few examples of
badly behaved smooth hypersurfaces, would be an important first step toward a better 
understanding of the theta body hierarchy for smooth varieties. 
Non-compact varieties have also not received much attention.

\section{More Examples and Applications}

\subsection{The maximum cut problem}

In section \ref{sec:stableset} we computed the theta body hierarchy 
for the maximum stable set problem. Another important problem in 
combinatorial optimization is the maximum cut (maxcut) problem. 
Given a graph $G=([n],E)$, we say that $C \subseteq E$ is a {\bf cut} of $G$, 
if there exists some partition of the vertices into two sets $V$ and $W$ 
such that $C$ is the set of edges between a vertex in $V$ and a vertex in $W$. 
The maxcut problem is the problem of finding a cut in $G$ of maximum cardinality. 
Theta bodies of the maxcut problem were studied in \cite{GLPT}, and in this 
subsection we will present some of those results.

The model we use is similar to the one used in the stable set problem. The characteristic vector of a cut $C$ of $G$, 
is the vector $\chi_C \in \{-1,1\}^{E}$ defined as $\chi_C(e)=-1$ if $e \in C$ and $1$ otherwise. Let $C_G$ be the 
collection of all characteristic vectors of cuts of $G$, and $\I(C_G)$ the vanishing ideal of $C_G$. 
The maximum cardinality cut
of $G$ can be found, in principle, by optimizing $\sum_{e \in E} x_e$ over $C_G$. However, this is a difficult problem 
and the size of the the maxcut can 
be approximated by optimizing $\sum_{e \in E} x_e$ over $\TH_k(\I(C_G))$. 

It is not hard to show that $\I(C_G)$ is generated by the polynomials 
$1-x_e^2$ for all $e\in E$, together with the polynomials
$1-\prod_{e\in K} x_e$ where $K \subseteq E$ is a chordless cycle. 
Using this one can construct a $\theta$-basis for $\I(C_G)$, but to do that
we need to introduce another combinatorial concept. 
Given an even set $T \subseteq [n]$, we call a subgraph 
$H$ of $G$ a $T$-join if the set of vertices
of $H$ with odd degree is precisely $T$. For example an 
$\emptyset$-join is a cycle, and the minimal $\{s,t\}$-join is the shortest path from 
$s$ to $t$. It is clear that there exists a $T$-join if and only if 
$T$ has an even number of nodes in each of the connected components of $G$. 
Define $\T_G :=\{T\subseteq [n] \,:\, \exists \,\, T-\textrm{join in } G\}$, and 
for each $T \in \T_G$, choose $H_T$ to be a $T$-join
with a minimal number of edges. Define 
$\T_k := \{T \subseteq \T_G : |H_T| \leq k\}$, and note that $\T_1=\{ \emptyset \} \cup E$. 
Then one can show that
$\B=\{\prod_{e \in T} x_e + \I(C_G) \,:\, T \in \T_G\}$ is a $\theta$-basis for $\I(C_G)$, and 
we can therefore identify $\B$ with $\T_G$ and each $\B_k$ with $\T_k$.

We can now give a description of the theta bodies of $\I(C_G)$.
$$\TH_k(\I(C_G)) = \left\{ \yy \in \RR^E \,:\,
\begin{array}{l} 
  \exists \, M \succeq 0, \, M \in \RR^{|\T_k| \times |\T_k|}
  \,\textup{such that} \\ 
  M_{e \emptyset} = y_e ,\,\,\forall\,\,e \in E\\
  M_{T T} = 1,\,\,\forall\,\,T \in \T_k\\
  M_{T_1 T_1} = M_{T_3 T_4} \,\,\textup{if} \,\, T_1 \Delta T_2 = T_3 \Delta T_4 
\end{array}
\right \}.$$ 
In particular, since $\T_1 = \{\emptyset\} \cup E$ we get 
$$\textup{TH}_1(IC_G) = \left\{ \yy \in \RR^E \,:\,
\begin{array}{l} 
\exists \, M \succeq 0, M \in \RR^{\{\emptyset\} \cup E \times \{\emptyset\} \cup E}\,\textup{such that} \\ 
M_{\emptyset \emptyset} = M_{e e} = 1,\,\,\forall\,\,e \in E\\
M_{\emptyset e} = M_{e \emptyset} =  y_e \,\,\forall\,\,e \in E\\
M_{e f} = y_g \,\,\textup{if} \,\, \{e,f,g\} \,\, \textup{is a triangle in} \,\,G \\
M_{e f} = M_{g h}  \,\,\textup{if} \,\, \{e,f,g,h\} \,\, \textup{is a square in} \,\,G
\end{array}
\right \}.$$
These relaxations are very closely related to some older 
hierarchies of relaxations for the maxcut problem. 
In particular they have a very interesting relation with the relaxation
introduced in \cite{LaurentMaxCut}. 
The theta body relaxation as it is, is not very strong, 
as shown by the following proposition.

\begin{proposition}[{\cite[Corollary 5.7]{GLPT}}]
Let $G$ be the $n$-cycle. Then the smallest $k$ for which $\I(C_G)$ is $\TH_k$-exact is $k =\lceil n/4 \rceil$.
\end{proposition}

Convergence can be sped up for graphs with large cycles 
by taking a chordal closure of the graph, computing the theta 
body for this new graph, and then
projecting to the space of edges of the original graph. If instead of the chordal closure 
one takes the complete graph over the same vertices, we essentially recover
the hierarchy introduced in \cite{LaurentMaxCut}.

One can use the usual theta body hierarchy together with Theorem \ref{thm:exact_points} 
to solve a question posed by Lov{\'a}sz 
in \cite[Problem 8.4]{Lovasz}. Motivated by the fact that a graph is perfect 
if and only if the stable set ideal $I_G$
is $\TH_1$-exact, Lov{\'a}sz asked to characterize
the graphs that were \textquotedblleft cut-perfect\textquotedblright , 
i.e., graphs $G$ for which $\I(C_G)$ is $\TH_1$-exact. It turns out
that such a \textquotedblleft strong cut-perfect 
graph theorem\textquotedblright \ is not too hard to derive.

\begin{proposition}[{\cite[Corollary 4.12]{GLPT}}]
Given a graph $G$, the ideal $\I(C_G)$ is $\TH_1$-exact if and 
only if $G$ had no $K_5$ minor and no chordless circuit of length 
greater than or equal to five.
\end{proposition}

\subsection{Permutation Groups}

In \cite[Sect. 4]{DHMO}, De Loera et al. study the behavior of the theta body hierarchy 
when applied to ideals associated to permutation groups.
Let $A \subseteq S_n$ be a subgroup of permutations and identify each 
element in $A$ with a permutation matrix.  We can then see these permutation
matrices as vectors in $\RR^{n \times n}$ and define $I_A$ to 
be their vanishing ideal. One of the main goals in \cite[Sect. 4]{DHMO} is to provide sufficient
conditions for the $\TH_1$-exactness of $I_A$.

A permutation group $A$ is \emph{permutation summable} 
if for all $P_1,\ldots,P_m \in A$ such that all entries of $\sum P_i - I$ are 
nonnegative, $\sum P_i - I$ is itself a sum of matrices in $A$. 
For example, $S_n$ itself is permutation summable as a direct consequence of
Birkhoff's Theorem. We say that $A$ is \emph{strongly fixed-point free} 
if the only permutation in $A$ that has fixed points is the identity.
In terms of matrices this is equivalent to saying that an element of $A$ 
is either the identity or has a zero diagonal.  This trivially
implies that all strongly fixed-point free groups are permutation summable. 

\begin{proposition}[{\cite[Theorem 4.5]{DHMO}}]
If $A\subseteq S_n$ is a permutation summable group, then $I_A$ is $\TH_1$-exact.
\end{proposition}

A very interesting class of permutation groups is the automorphism groups of graphs. 
Given a graph $G=([n],E)$, $\Aut(G) \subseteq S_n$ is defined
to be the group of vertex permutations $\sigma$ such that $\{i,j\} \in E$ 
implies $\{\sigma(i),\sigma(j)\} \in E$. The \emph{automorphism polytope} 
of $G$, denoted as $P_{\Aut(G)}$,
is the convex hull of $\Aut(G)$, where again we are identifying a 
permutation with its matrix representation. There has
been some study of ways to relax this polytope. See \cite{tinhofer} for instance.
One such relaxation 
is the polytope $P_G$ of all the points $P \in \RR^{n \times n}$ such that
$$PA_G = A_GP; \ \ \sum_{i=1}^n P_{i,j}=1, \ 1 \leq j \leq n;$$ 
$$\sum_{j=1}^n P_{i,j}=1, \ 1 \leq i \leq n; \ \ P_{i,j} \geq 0, \ 1 \leq i,j \leq n;$$
where $A_G$ is the $n \times n$ adjacency matrix of $G$. 
If in the constraints of $P_G$ we replace $P_{i,j} \geq 0$ by $P_{i,j}\in \{0,1\}$,
then we get precisely $\Aut(G)$. In fact, $P_{\Aut(G)}$ is the integer hull of $P_G$ i.e., the convex hull 
of the integer points of $P_G$. If $P_G=P_{\Aut(G)}$ then $G$ is said to be \emph{compact}. 
It is not hard to show that $\TH_1(I_{\Aut(G)}) \subseteq P_G$, so the first theta body 
provides an approximation of $P_{\Aut(G)}$ that is at least as good as 
the linear approximation $P_G$, but one can actually say more. 

\begin{proposition} \cite[Theorem 4.4]{DHMO}
The class of compact graphs is strictly included in the class of graphs with $\TH_1$-exact automorphism 
ideal. In particular, let $G_1,\ldots,G_m$ be $k$-regular graphs (all vertices have degree $k$) that are compact 
and $G$ their disjoint union. 
Then $I_{\Aut(G)}$ is $\TH_1$-exact, but $G$ is compact if and only if $G_1\cong \cdots\cong G_m$.
\end{proposition}

Finally note that while theta bodies of automorphism groups of graphs 
have interesting theoretical properties, computing with them tends to
be quite hard, as there is no easy general way to obtain a good 
$\theta$-basis for this problem. In fact, knowing if a graph has a non-trivial
automorphism, the \emph{graph automorphism problem}, is equivalent to knowing 
if a $\theta$-basis for $I_{\Aut(G)}$
has more than one element. Determining the complexity class of this problem is a 
major open question.

%%%%%%%%%%%%%%%%%%%%%%%%%%%%%%%%%%%%%%%%%%%%%%%%%%%%%%%%%%%%%%%%%%%%%%

\bibliographystyle{plain}

\printindex
\end{document}